\documentclass[a4paper, 11pt, twoside]{article}
\usepackage{amsmath,amsthm}
\usepackage{amssymb,latexsym}
\usepackage{mathrsfs}
\usepackage{enumerate}
\usepackage[colorlinks,citecolor=blue,dvipdfm]{hyperref}
\usepackage[numbers,sort&compress]{natbib}
\usepackage{indentfirst}
\DeclareMathAlphabet{\mathpzc}{OT1}{pzc}{m}{it}

\newtheorem{theorem}{Theorem}[section]
\newtheorem{corollary}{Corollary}[section]
\newtheorem{lemma}{Lemma}[section]

\newtheorem{definition}{Definition}[section]
\newtheorem{remark}{Remark}[section]

\theoremstyle{definition} \theoremstyle{remark}
\numberwithin{equation}{section}
\allowdisplaybreaks

\begin{document}
	\markboth{X. Huang, L. Peng}{Local existence and nonexistence of time nonlocal evolution equations}
	
	\date{}
	\baselineskip 0.22in
	\title{{\bf The existence and nonexistence for time nonlocal evolution equations with superlinear sources }}
	
	\author{Xi Huang$^{1}$, Li Peng$^{1}$\thanks{\footnotesize {Corresponding author: lipeng\_math@xtu.edu.cn.}} \\[1.8mm]
		\footnotesize  {$^1$Faculty of Mathematics and Computational Science, Xiangtan University}\\
		\footnotesize  {Hunan 411105, China}
	}
	
	\maketitle
	
	\begin{abstract}
		We consider the solvable behaviour for the initial value problem of time nonlocal evolution equations with the kernels of type $(\mathcal{PC})$. Our aim is to analyze some sufficient conditions ensuring local existence, integrability of mild solutions when the nonlinear term exhibits rapid growth. Moreover, a sufficient condition of the unsolvable result is also established. It turns out that the solvable behaviour is closely connected with the index on the initial value, which occurs a critical dimension phenomenon. The proofs rely on subordination and monotone iterative method. Finally, several examples are given to illustrate the wide applicability of the results.\\[2mm]
		
	\end{abstract}
	
	\section{Introduction}
	The memory integral operator serves as an efficient mathematical tool for accurately characterizing complex memory-dependent rheological and thermal behaviors in areas such as viscoelastic heat conduction and transport in disordered media. For instance, in certain non-Fickian heat conduction processes, let $u(t,x)$ denote the temperature in a rigid domain $\Omega \subset \mathbb{R}^3$, and let $m \in BV_{\mathrm{loc}}(\mathbb{R}_+)$ characterize the memory effect of the system. The constitutive relation for the internal energy density is given by
	\begin{align*}
		\epsilon(t,x) = \int_0^\infty u(t-\theta, x)\, \mathrm{d}m(\theta).
	\end{align*}
	Recently, in \cite{V. Vergara}, combining the classical Fourier law $q(t,x) = -c_0 \nabla u(t,x)$ with the energy balance law
	\begin{align*}
		\partial_t \epsilon(t,x) + \operatorname{div} q(t,x) = h(t,x),
	\end{align*}
	and considering the case $m(t) = \int_0^t k(\theta)\, \mathrm{d}\theta$, a simplified model for heat conduction with memory has been established, which takes the form
	\begin{align}\label{1.1}
		\partial_t (k \ast [u-u_0])(t, x) - c_0 \Delta u(t, x) = h(t, x), \quad t>0, \ x \in \Omega \subset \mathbb R^N,
	\end{align}
	where $c_0 > 0$ is the thermal conductivity, $\Delta$ is the Laplacian operator and the convolution $\ast$ is defined by $(k \ast v)(t)=\int_0^t k(t-\theta)v(\theta)\, \mathrm{d}\theta$ for $t \geq 0$. The kernel $k$ satisfies the following condition:
	\begin{itemize}
		\item[$(\mathcal{PC})$] $k \in L^1_{\mathrm{loc}}(\mathbb R_+)$ is nonnegative and nonincreasing, and there exists a kernel $l \in L^1_{\mathrm{loc}}(\mathbb R_+)$ such that $(k \ast l)(t) = 1$ for $t > 0$. We also record this situation as $(k,l) \in (\mathcal{PC})$.
	\end{itemize}
	
	Equation \eqref{1.1} covers both the time-fractional diffusion equation and the distributed-order fractional diffusion equation, which have been the subject of extensive study, see \cite{A. Agresti, H. Kozono, A.N. Kochubei, K.M. Owolabi}. For the nonlocal evolution equations with the kernels of type $(\mathcal{PC})$, sharp decay estimates for solutions to a class of nonlocal equations on a bounded domain were derived by Vergara and Zacher \cite{V. Vergara}, whose arguments employed energy methods and a novel inequality for integro-differential operators. Subsequently, using Fourier multiplier methods and relaxation functions theory, Kemppainen et al. \cite{J. Kemppainen} obtained the $L^r$ estimates for solutions to \eqref{1.1} with $f=0$, and then they claimed that the decay profile of solutions exhibits a critical dimension phenomenon. Very recently, Pozo and Vergara \cite{J.C. Pozo} used the subordination principle to construct the fundamental solutions of fully non-local equations, and
	large-time behavior of solutions was found by using harmonic analysis and
	Fourier multiplier methods.
	
	This paper focuses on a class of time-nonlocal diffusion equations with a $(\mathcal{PC})$ kernel and a nonlinear term $f(u)$ that exhibits exponential or even super-exponential growth. Such equations accurately capture complex anomalous transport phenomena characterized by strong memory effects, including solute transport in highly disordered porous media and heat conduction in viscoelastic materials with long-range temporal correlations. Their modeling capability stems from two core mechanisms: the $(\mathcal{PC})$ kernel reproduces the inherent history memory effect in non-Fickian transport, while the rapidly growing nonlinear term precisely characterizes the intense energy accumulation driven by strong external fields and the potential occurrence of finite-time blow-up. The combination of these two mechanisms gives rise to dissipative mechanisms and blow-up threshold conditions that are absent in classical parabolic equations with analogous nonlinearities. Accordingly, we study the following initial value problem for time-nonlocal evolution equations:
	\begin{equation}\label{1.2}
		\begin{cases}
			\partial_t\left[k\ast(u-u_0)\right]-\Delta u= f(u), \quad 0<t<T,\ x\in \mathbb R^N,\\
			u(0,x)=u_0(x),\quad x\in \mathbb R^N,
		\end{cases}
	\end{equation}
	where $u_0$ is the initial condition, $(k,l) \in (\mathcal{PC})$ and $f$ satisfies the following rapid growth condition:
	\begin{itemize}
		\item[$(A_f)$] $f \in C^1[0, \infty)$ satisfies$ f(u)>0, \ f'(u)>0 ~{\rm and}~ G(u):=\int_{u}^{\infty} \frac{1}{f(\tau)} \mathrm{d}\tau< \infty ~{\rm for}~ u>0$, and
		moreover, $\lim\limits_{u \to \infty}f'(u)G(u)=1$.
	\end{itemize} 
	The condition $\lim\limits_{u \to \infty}f'(u)G(u)=1$ characterizes functions whose growth rate is super-polynomial. Prominent examples satisfying $(A_f)$ include $f(\tau)=e^{\tau^p}$ ($p\geq 1$) and $f(\tau)=e^{e^{\tau}}$, see \cite{M. Suzuki 19}.
	
	In the papers mentioned above, it was shown that the work on nonlocal evolution equations with the kernels of type $(\mathcal{PC})$ was mainly focused on asymptotic profile of solutions for the case that $f$ is linear, there are very limited literatures dealing with solvability in the case of nonlinearity.
	Superlinearly growing source terms introduce a strong self-exciting positive feedback mechanism that competes intensely with diffusion, dissipation, and memory effects. Investigating the conditions for local existence and non-existence of solutions to such equations not only reveals the critical balance among source term growth, medium properties, and initial data regularity, thereby refining the well-posedness theory, but also characterizes whether the physical system can evolve reasonably in the initial stage and identifies the critical thresholds for instantaneous singularity formation, loss of control, or even system breakdown. For instance, some results on the well-posedness, asymptotic behavior and  blow-up for the solutions of classical heat equations and time-fractional diffusion equations were obtained when the nonlinear source $f$ is superlinear, see \cite{Y. Fujishima, N. Ioku, M. Suzuki 19, Y. Xie}. Most recently, \cite{M. Suzuki 22} studied the existence conditions for local positive solutions of time-fractional diffusion equations with rapidly growing source terms using the monotone iterative method. As a special case of nonlocal evolution equations, that work highlighted the crucial role of the Gaussian heat kernel and the Wright function $\Phi_\alpha$ in the proofs. Inspired by this idea, we introduce the propagation function $w(t, \tau)$ for $0 < \tau, t < \infty$ associated with the completely positive kernel $l$ (cf. \cite{J. Kemppainen}) and combine the monotone iterative method with relaxation function theory to establish the local existence of positive solutions to \eqref{1.2} in the locally uniform space $L_{ul}^p(\mathbb R^N)$. We also discuss the non-existence of positive solutions under certain conditions. 
	
	The article is structured as follows. In Section 2, we introduce the necessary notations, the concepts of certain locally uniform spaces, and the relaxation function theory. Section 3 is devoted to establishing the existence of a supersolution and presenting several $L_{ul}^p$–$L_{ul}^q$ estimates. The local existence and nonexistence of solutions to \eqref{1.2} are investigated in Section 4. Finally, in Section 5, we provide several examples of pairs $(k,l)$ of type $(\mathcal{PC})$ to illustrate our results, using the Karamata–Feller Tauberian theorem.
	
	\section{Preliminaries}
	Let $1 \leq p <\infty$. We recall the locally uniform space $L_{ul}^p(\mathbb R^N)$ is defined as
	$$
	\left\{u\in L_{loc}^{p}\left(\mathbb R^{N}\right):\left\|u\right\|_{L_{ul}^p\left(\mathbb R^{N}\right)}:=\sup_{z\in \mathbb R^{N}}\left(\int_{B\left(z\right)}\left|u(x)\right|^{p}dx\right)^{\frac{1}{p}}<\infty\right\},
	$$
	where $B(z)$ is the open ball in $\mathbb R^N$ with center $z$ and radius $1$.	We set $L_{ul}^{\infty}(\mathbb R^N) = L^{\infty}(\mathbb R^N)$ for $p=\infty$.
	Let $BUC(\mathbb R^N)$ be the space of bounded uniformly continuous functions and
	$\mathcal{L}_{ul}^p\left(\mathbb{R}^{N}\right)$ be the closure of $BUC(\mathbb R^N)$ in the space $L_{ul}^p(\mathbb R^N)$, that is,
	$$
	\mathcal L_{ul}^{p}(\mathbb R^N): =\overline{BUC(\mathbb R^N)}^{\|\cdot \|_{L_{ul}^{p}(\mathbb R^N)}}.
	$$
	
	For $\mu \geq 0$, we recall the scalar integral equations
	\begin{align}
		&s_{\mu}(t) + \mu(l \ast s_\mu)(t)=1,\quad t \geq 0,\label{eq 2.1} \\
		&r_{\mu}(t) + \mu(l \ast r_\mu)(t)=l(t),\quad t>0.\label{eq 2.2}
	\end{align}
	For any $l \in L_{\mathrm{loc}}^1(\mathbb R_+)$, \eqref{eq 2.1} and \eqref{eq 2.2} admit unique solutions $s_\mu$ and $r_\mu \in L_{\mathrm{loc}}^1(\mathbb R_+)$, respectively, see \cite[Chapter 2, Theorem 3.1]{G. Gripenberg}.
	If $l \in L_{loc}^1(\mathbb R_+)$ such that $s_\mu$ is nonnegative and nonincreasing with respect to $t>0$ for every $\mu>0$, then $l$ is referred to as a completely positive function, see \cite[Proposition 4.5]{J. Pruss}.
	Furthermore, $(k,l)\in (\mathcal{PC})$ implies $l$ is completely positive and nonnegative, see \cite{P. Clement 81}.
	We summarize below some properties of  $s_\mu$ and $r_\mu$ from \cite{P. Clement 79, J.C. Pozo}(c.f. \cite[Lemma A.1]{X. Huamg}).
	\begin{lemma}\label{lemma 2.1}
		Assume that $(k,l) \in (\mathcal{PC})$. Then the following proposition holds.
		\begin{itemize}
			\item[(i)] $s_\mu(t)$ admits the representation $s_\mu(t) = -\int_0^\infty e^{-\mu \tau} w(t,\mathrm{d}\tau), \ t\geq 0$,
			where $-w(t, \mathrm{d}\tau)$ is a positive finite measure such that $-\widehat{w}(\lambda, \mathrm{d}\tau) = (\lambda \widehat{l}(\lambda))^{-1} e^{-\tau \widehat {l}(\lambda)^{-1}}\mathrm{d}\tau$. Moreover, $-\int_0^\infty w(t,\mathrm{d}\tau)=1$ for $t \geq 0$. 
			\item[(ii)] $r_\mu(t)$ admits the representation $r_\mu(t) = \int_0^\infty e^{-\mu \tau} \eta(t,\mathrm{d}\tau),\ t>0$,
			where $\eta(t, \mathrm{d}\tau)$
			is a positive finite measure  such that $\widehat{\eta}(\lambda, \mathrm{d}\tau) = e^{-\tau \widehat{l}(\lambda)^{-1}}\mathrm{d}\tau$. Moreover, $-\int_0^\infty \eta(t,\mathrm{d}\tau)=l(t)$ for $t>0$. 
		\end{itemize}
	\end{lemma}
	
	Next, we introduce the function $Z(t,x)$ by
	\begin{align}\label{eq 2.3}
		Z(t,x)=-\int_{0}^{\infty} H(\tau, x) w(t,\mathrm{d}\tau), \quad t>0,\ x \in \mathbb R^N,
	\end{align}	
	where $H(\cdot,\cdot)$ is the Gaussian Heat Kernel.
	It follows from \cite{J.C. Pozo} that the formula of mild solutions of (\ref{1.2}) is given by
	\begin{align*}
		u(t,x)=Z(t,\cdot)\ast_x u_0+\int_{0}^{t} Y(t-\theta,\cdot)\ast_x f(u(\theta,\cdot))\, \mathrm{d}\theta,
	\end{align*}
	where the symbol $\ast_x$ represents the convolution in $\mathbb R^N$, and the function $Y$ is defined by
	\begin{align}\label{eq 2.4}
		Y(t, x)=\int_{0}^{\infty} H(\tau, x) \eta(t, \mathrm{d}\tau), \quad t>0, \ x \in \mathbb R^N.
	\end{align}
	Moreover, since $H(\tau, x)$ is nonnegative on $(0, \infty)\times \mathbb R^N$, then (\ref{eq 2.3}) and (\ref{eq 2.4}) show that $Z(t, x)$ and $Y(t, x)$ are nonnegative.
	
	\begin{remark}\label{remark 2.2}
		(\cite[Proposition 1]{J.C. Pozo}) The functions $Z(t, x)$ and $Y(t,x)$ satisfy
		\begin{align}\label{eq 2.5}
			\int_{\mathbb R^N} Z(t,x)\, \mathrm{d}x=1 ~~{\it and}~
			\int_{\mathbb R^N} Y(t,x)\, \mathrm{d}x=l(t) ~~{\it for}~ t>0.
		\end{align}
	\end{remark}
	Now, we give the definition of the mild solution of problem (\ref{1.2}).
	\begin{definition}\label{def 2.1}
		Let $T>0$. A measurable function u becomes the mild solution of (\ref{1.2}) if and only if it satisfies the following conditions:
		\begin{enumerate}
			\item[{\rm(i)}]
			$u(t) = u(t,\cdot) \in L_{ul}^1(\mathbb R^N)$ for $t \in (0,T)$.
			\item[{\rm(ii)}]
			$f(u(t))=f\big(u(t,\cdot)\big)\in L_{ul}^1(\mathbb R^N)$ for almost all $t\in (0,T)$.
			\item[{\rm(iii)}]
			For $t\in (0,T)$ and a.e. $x\in \mathbb R^N$, it holds that
			$$
			u(t,x)=Z(t,\cdot)\ast_x u_0 +\int_{0}^{t} Y(t-\theta,\cdot)\ast_x f\big(u(\theta,\cdot)\big)\, \mathrm{d}\theta.
			$$
			\item[{\rm(iv)}]
			For $t\in (0,T)$, we have
			\begin{equation}\label{eq 2.10}
				\int_{0}^{t} {\Big \|f\big(u(s)\big) \Big\|}_{L_{ul}^1(\mathbb R^N)}\, \mathrm{d}s<\infty.
			\end{equation}
			\item[{\rm(v)}]
			The initial function $u_0$ satisfies
			\begin{equation}\label{eq 2.11}
				{\big\|u(t)-Z(t,\cdot)\ast_x u_0\big\|}_{L^{\infty}(\mathbb R^N)}\to 0 {~~{\it as}~t \to 0.}
			\end{equation}
		\end{enumerate}
	\end{definition}
	
	Throughout this paper, we adopt the following conventions. For any set $Y$, let $I_i: Y \to [0,\infty)$ for $i=1,2$. We write $I_1 \lesssim I_2$ on $Y$ if there exists $C>0$ such that $I_1(z) \leq C I_2(z)$ for all $z \in Y$. For two functions $g$ and $h$, we denote $g \asymp h$ as $t \to t_0$ if there exist $C_1, C_2 > 0$ such that $C_1 g(t) \leq h(t) \leq C_2 g(t)$ near $t_0$, and $g \sim h$ as $t \to t_0$ if $\lim\limits_{t \to t_0} \frac{g(t)}{h(t)} = C$ for some $C>0$.
	
	\section{Technical tools}
	In this section, we prove several lemmas involving the estimates of the solution for (\ref{1.2}) and convex functions. We start with an introduction of the monotone iterative method.
	\begin{lemma}\label{lemma 3.1}
		Let $u_0\geq 0$, $T >0$. If $f$ satisfies $(A_f)$, and there exists a function $\psi\in L_{loc}^ {\infty}\big((0,T),L^{\infty}(\mathbb R^N) \big)$ with $\psi(t)\geq0$ satisfying
		$$
		\psi(t,x) \geq \mathcal A\left[\psi(t)\right]:=Z(t,\cdot)\ast_x u_0+\int_{0}^{t} Y(t-\theta,\cdot)\ast_x f(\psi(\theta,\cdot))\, \mathrm{d}\theta
		$$
		for $t \in (0, T)$ and $a.e.~ x \in \mathbb R^N$, then the equation $u(t)=\mathcal A\left[u(t)\right]$ has at least one solution $u$ satisfying $0\leq u(t,x)\leq \psi(t,x)$ for $t \in (0, T)$ and $a.e.~x \in \mathbb R^N$.
		\begin{proof}
			Define $\psi_1(t)=\psi_1(t,x)=Z(t,\cdot)\ast_x u_0$ and $\psi_j(t)=\psi_j(t,x)=\mathcal A\big[\psi_{j-1}(t)\big]$ for $j\in \mathbb Z_{j\geq 2}$. Since $Y$ is nonnegative, then
			$\psi_2(t,x) \geq \psi_1(t,x)$
			for  $t \in (0, T)$ and a.e. $x \in \mathbb R^N$. Assume that for $j \in \mathbb N_{j\geq 2}$,
			$$
			\psi_j(t,x)\geq \psi_{j-1}(t,x), \quad t\in (0,T), ~{ \rm a.e.}~ x\in \mathbb R^N.
			$$
			Then, since $f$ is monotonically increasing, we 
			$$
			\psi_{j+1}(t)-\psi_j(t)=\int_{0}^{t} Y(t-\theta, \cdot)\ast_x \Big[f\big(\psi_j(\theta)\big)- f\big(\psi_{j-1}(\theta)\big)\Big]\, \mathrm{d}\theta\geq 0,
			$$
			By induction, we obtain
			$$0 \leq \psi_1(t,x) \leq \cdots \leq \psi_j(t,x) \leq \cdots \leq \psi(t,x) ~~{\rm for}~ t\in (0,T),~{\rm a.e.}~x \in \mathbb R^N.
			$$
			According to the monotonic convergence principle, $\lim\limits_{j \to \infty}\psi_j(t,x)$ exists for $t \in (0,T)$ and a.e. $x \in \mathbb R^N$.
			Define $u(t,\cdot):=\lim\limits_{j\to\infty}\psi_j(t,\cdot)$, then
			$$
			\lim\limits_{j \to \infty}\mathcal A[\psi_{j-1}]=\mathcal A[u],
			$$
			hence $u=\mathcal A[u]$ and $0 \leq u(t,x) \leq \psi(t,x)$ for $t \in (0,T)$ and for $t \in (0,T)$ and a.e. $x \in \mathbb R^N$,  $x \in \mathbb R^N$.
		\end{proof}
	\end{lemma}
	
	Below we state the standard $L_{ul}^p-L_{ul}^q$ estimate of the operator $H(t,\cdot)\ast_x$, see \cite[Proposition 2.1]{J.M. Arrieta}.
	\begin{lemma}\label{lemma 3.2}
		Let $N\geq 1$, $1\leq p\leq q \leq \infty$ and $(k,l)\in (\mathcal {PC})$. Then for any $v \in L_{ul}^p(\mathbb R^N)$, one can find a constant $M>0$, depending only on $N$, such that
		\begin{equation*}
			\|H(t,\cdot)\ast_x v\|_{L_{ul}^q(\mathbb R^N)} \leq M\big[(1\ast l)(t)\big]^{-\frac{N}{2}(\frac{1}{p}- \frac{1}{q})}\|v\|_{L_{ul}^p(\mathbb R^N)}, \quad t>0.
		\end{equation*}
	\end{lemma}
	
	The following lemma introduces a weighted Stieltjes integral that provides an upper bound for the $L_{ul}^p$–$L_{ul}^q$ estimate of the solution operator $Z(t,\cdot)\ast_x$. 
	It follows from \cite[Lemma 2]{J. C. Pozo 24} (see also \cite[Lemma A.2]{X. Huamg}).
	
	\begin{lemma}\label{lemma 3.3}
		Let $\delta \in [0,1)$, $(k,l)\in (\mathcal{PC})$. Then the following two functions
		\begin{align*}
			a_{\delta}(t):=-\int_0^\infty \tau^{-\delta}\, w(t, \mathrm{d} \tau), \ b_{\delta}(t):=-\int_0^\infty \tau^{-\delta}\, \eta(t, \mathrm{d} \tau), \quad t>0,
		\end{align*}
		are locally $L^1$-integrable on $\mathbb{R}+$. Their Laplace transforms are given by
		\begin{align*}
			\widehat{a}_{\delta} (\lambda)=\Gamma(1-\delta) \lambda^{-1} \widehat{l}(\lambda)^{ -\delta} , \ \widehat{b}_{\delta} (\lambda)=\Gamma(1-\delta) \widehat{l}(\lambda)^{ 1-\delta}.
		\end{align*}
		Moreover, it holds that $a_{\delta}(t)\leq \Gamma(1-\delta) [(1 \ast l)(t)]^{-\delta}$ for $t>0$, where $\Gamma(\cdot)$ denotes the Gamma function. If $l$ is also nonincreasing, then  $b_{\delta}(t)\leq \Gamma(1-\delta) l(t)[(1 \ast l)(t)]^{-\delta}$ for $t>0$.
	\end{lemma}
	
	Now we present the $L_{ul}^p-L_{ul}^q$ estimates of the operator $Z(t,\cdot)\ast_x$.
	\begin{lemma}\label{lemma 3.4}
		Let $N\geq 1$, $1\leq p \leq q\leq\infty$ such that $\frac{N}{2}(\frac{1}{p}- \frac{1}{q})\in [0,1)$, and $(k,l)\in (\mathcal {PC})$. Assume that $v \in L_{ul}^p(\mathbb R^N)$, then there exist constants $M_1, M_2>0$, depending only on $p$, $q$, and $N$, such that
		\begin{align}
			&\|Z(t,\cdot)\ast_x v\|_{L_{ul}^q(\mathbb R^N)} \leq
			M_1\Big(\big[(1\ast l)(t)\big]^{-\frac{N}{2}(\frac{1}{p}- \frac{1}{q})}+ 1\Big)\|v\|_{L_{ul}^p(\mathbb R^N)}, \quad t>0, \label{eq 3.1}  \\
			& \|Z(t,\cdot)\ast_x v\|_{L_{ul}^q(\mathbb R^N)} \leq M_2\big[(1\ast l)(t)\big]^{-\frac{N}{2}(\frac{1}{p}- \frac{1}{q})}\|v\|_{L_{ul}^p(\mathbb R^N)} \label{eq 3.2}
		\end{align}
		for sufficiently small $t>0$.
		\begin{proof}
			From \eqref{eq 2.3}, Lemma \ref{lemma 3.2}, and Lemma \ref{lemma 3.3}, we obtain
			\begin{align*}
				\|Z(t,\cdot)\ast_x v\|_{L_{ul}^q(\mathbb R^N)} &\leq -\int_0^\infty \|H(\tau,\cdot)\ast_x v\|_{L_{ul}^q(\mathbb R^N)}\, w(t,\mathrm{d}\tau)\\
				&\leq -M\|v\|_{L_{ul}^p(\mathbb R^N)} \int_{0}^\infty \tau^{-\frac{N}{2}(\frac{1}{p}- \frac{1}{q}) } \, w(t,\mathrm{d}\tau)\\
				&\leq M\Gamma(1 - \frac{N}{2}(\frac{1}{p}- \frac{1}{q}) )[(1\ast l)(t)]^{-\frac{N}{2}(\frac{1}{p}- \frac{1}{q})}\|v\|_{L_{ul}^p(\mathbb R^N)} \text{ ~for~ } t>0.
			\end{align*}
			Setting $M_1 = M\Gamma \left(1 - \frac{N}{2}\left(\frac{1}{p} - \frac{1}{q}\right)\right)$, we immediately obtain \eqref{eq 3.1}. Choose $T_0$ such that $(1 \ast l)(T_0) = 1$ and let $M_2 = 2M_1$. Then \eqref{eq 3.2} holds for $t \in (0, T_0]$.
		\end{proof}
	\end{lemma}
	
	Inspired by \cite[Lemma 8]{H. Brezis}, we obtain the following lemma.
	\begin{lemma}\label{lemma 3.5}
		Let $N\geq 1$, $1\leq p\leq q\leq\infty$ such that $ \frac{N}{2}(\frac{1}{p}- \frac{1}{q}) \in [0,1)$, and $(k,l)\in (\mathcal {PC})$. For every fixed $v \in \mathcal L_{ul}^p(\mathbb R^N)$, then there is a nonnegative and increasing function $C^{\ast}(t)$ for $t \in (0,1)$ with $\lim\limits_{t\to 0}C^{\ast}(t)=0$, such that
		\begin{align}\label{eq 3.3}
			\big\|Z(t,\cdot)\ast_x v \big\|_{L_{ul}^q(\mathbb R^N)}
			\leq C^{\ast}(t)\big[(1\ast l)(t) \big]^{-\frac{N}{2}(\frac{1}{p} - \frac{1}{q})}
		\end{align}
		for sufficiently small $t>0$.
		\begin{proof}
			We employ the dense approximation method. 
			Define $\varpi:(0, \infty) \rightarrow(0, \infty)$ by
			$$
			\varpi(t)=\big[(1\ast l)(t) \big]^{\frac{N}{2}(\frac{1}{p} - \frac{1}{q})}.
			$$
			From Lemma \ref{lemma 3.3}, Remark \ref{remark 2.2} and Young's inequality, we can get
			\begin{align*}
				\varpi(t) \big\|Z(t,\cdot)\ast_x v\big\|_{L_{ul}^q(\mathbb R^N)}
				\leq& \varpi(t) \big\|Z(t,\cdot)\ast_x (v-v_1)\big\|_{L_{ul}^q(\mathbb R^N)}\\
				&+ \varpi(t) \big\|Z(t,\cdot)\ast_x v_1\big\|_{L_{ul}^q(\mathbb R^N)}\\
				\lesssim& \big\| (v-v_1)\big\|_{L_{ul}^{p}(\mathbb R^N)}
				+ \varpi(t) \|v_1\|_{L^{\infty}(\mathbb R^N)}
			\end{align*}
			for $v_1 \in L^{\infty}(\mathbb R^N)$ and $t>0$. It follows that
			$$
			\limsup\limits_{t\to 0}\varpi(t) \big\|Z(t,\cdot)\ast_x v\big\|_{L_{ul}^q(\mathbb R^N)}\leq \big\| v-v_1\big\|_{L_{ul}^{p}(\mathbb R^N)}$$
			for $v_1 \in L^{\infty}(\mathbb R^N)$. By the arbitrariness of $v_1$, we can get (\ref{eq 3.3}).
		\end{proof}
	\end{lemma}
	
	\begin{lemma}\label{lemma 3.6}
		Let $K:[C, \infty) \to [0, \infty)$ be a convex function for $C \geq 0$. If $v\in L_{ul}^1(\mathbb R^N)$, $v>C$ and $K(v)\in L_{ul}^1(\mathbb R^N)$, then
		\begin{align}\label{eq 3.4}
			K\big(Z(t,\cdot)\ast_x v\big) \leq Z(t,\cdot)\ast_x K(v)
		\end{align}
		and
		\begin{align}\label{eq 3.5}
			K\left(\frac{Y(t,\cdot)\ast_x v}{l(t)}\right) \leq \frac{Y(t,\cdot)\ast_x K(v)}{l(t)}
		\end{align}
		in $(0,\infty)\times \mathbb R^N$.
		\begin{proof}
			Let $C=0$. From Jensen's integral inequality and {Remark \ref{remark 2.2}}, we obtain
			\begin{align*}
				K\big(Z(t,\cdot)\ast_x v\big)
				=K\left(\int_{\mathbb R^N} Z(t, x-y) v(y)\, \mathrm{d}y \right)
				\leq Z(t,\cdot)\ast_x K(v).
			\end{align*}
			Assume $C> 0$, we put $\mathring{K}(v):=K(v+C)$ and $\mathring{v}:=v-C$. {It is easy to know that} $\mathring{K}:[C, \infty] \to [0, \infty)$
			is a convex function and $\mathring{v}\in L_{ul}^1(\mathbb R^N)$ satisfies $\mathring{v} \geq 0$, $K(\mathring{v})\in L_{ul}^1(\mathbb R^N)$. Then from the above inequality we deduce that
			$$
			\mathring K\big(Z(t,\cdot)\ast_x \mathring v\big) \leq Z(t,\cdot)\ast_x \mathring K(\mathring v),
			$$
			which yields
			\begin{align*}
				K\big(Z(t,\cdot)\ast_x (v-C)+ C\big) \leq Z(t,\cdot)\ast_x K(v).
			\end{align*}
			This along with $Z(t, \cdot)\ast_x C=C$ results in (\ref{eq 3.4}). In the same way, (\ref{eq 3.5}) can be proved. The proof is completed.
		\end{proof}
	\end{lemma}
	
	\section{Main Results}
	In this section, we give the local existence and nonexistence about the mild solution of problem (\ref{1.2}) with the kernels of type $(\mathcal{PC})$.
	\subsection{Existence}	
	Firstly, we present the existence result in the subcritical case $\rho > \frac{N}{2}$.
	\begin{theorem}\label{the 4.1}
		Let $N\geq 1$, $(k,l)\in (\mathcal {PC})$ and $u_0 \geq 0$. Suppose that $l$ is nonincreasing and $f$ satisfies $(A_f)$. If ${G(u_0)}^{-\rho}\in L_{ul}^1(\mathbb R^N)$ for some $\rho > \frac{N}{2}$, then (\ref{1.2}) admits a  nonnegative mild solution on the interval $[0, T)$ for a certain constant $T>0$.		
		\begin{proof}
			Take $c>\max\{\rho, 1\}$ to be a constant and define $K(u):={G(u)}^{-\frac{\rho}{c}}$. In what follows, we always take $C>0$ large enough.
			
			From $\lim\limits_{u \to \infty} f'(u)G(u)=1$ we obtain that the function $K(u)$ is convex and $f(u )/K(u)$ is nondecreasing when $u\geq C$. Moreover, $K^{-1}(u)$ is a concave function, $(K^{-1})(u )$ is increasing and $(K^{-1})'(u )$ is nonincreasing for $u \geq K(C)$.	
			Let $u_1(x):=\mathrm{max} \{u_0(x),C\}$, then $K(u_1)\in L_{ul}^c(\mathbb R^N)$. We define
			\begin{align*}
				\varphi(t):=K^{-1}\big((\delta+1)Z(t,\cdot)\ast_x K(u_1) \big),
			\end{align*}
			where $\delta >0$ is a constant.
			By Remark \ref{remark 2.2}, Lemma \ref{lemma 3.4} and the increasing property of $K^{-1}$ on $[K(C), \infty)$, we have
			\begin{equation}\label{eq 4.1}
				\begin{aligned}
					C&\leq
					\varphi(t) \\
					&\leq K^{-1}\Big((\delta+1)\big\|Z(t,\cdot)\ast_x K(u_1)\big\|_{L^{\infty}(\mathbb R^N)} \Big)\\
					&\leq K^{-1}\Big(M_2(\delta +1)\big[(1\ast l)(t) \big]^{-\frac{N}{2c}}\big\|K(u_1) \big\|_{L_{ul}^c(\mathbb R^N)} \Big)
					< \infty~~{\rm for}~ t\in(0,T_1),
				\end{aligned}
			\end{equation}
			which implies {$\varphi\in L_{loc}^{\infty}\big((0,T_1), L^{\infty}(\mathbb R^N) \big)$}, where $T_1$ is sufficiently small.  
			Then from Lemma \ref{lemma 3.6}, the monotonicity of $(K^{-1})'$ on $[K(C),\infty)$ and the mean value theorem, we obtain
			\begin{equation}\label{eq 4.2}
				\begin{aligned}
					\varphi(t) - Z(t,\cdot)\ast_x u_0
					&\geq \varphi(t) - Z(t,\cdot)\ast_x u_1\\
					&\geq \varphi(t) -K^{-1}\big(Z(t,\cdot)\ast_x K(u_1) \big)\\
					&\geq \frac{\delta Z(t,\cdot)\ast_x K(u_1)}{K'\Big( K^{-1}\big((\delta + 1)Z(t,\cdot)\ast_x K(u_1)\big)\Big)}\\
					&=\frac{\delta}{\delta + 1} \frac{K\big(\varphi(t)\big)}{K'\big(\varphi(t)\big)}\\
					&=\frac{c\delta}{\rho(\delta +1)} f\big( \varphi(t)\big) G\big(\varphi(t) \big).
				\end{aligned}
			\end{equation}
			
			We notice the following estimate holds:
			\begin{equation}
				\begin{aligned}\label{eq 4.3}
					&\quad \int_{0}^{t} Y(t-\theta,\cdot)\ast_x f\big(\varphi(\theta)\big)\,
					\mathrm{d}\theta \\
					&\leq \int_{0}^{t} \bigg\|\frac{f\big(\varphi(\theta) \big)}{K\big(\varphi(\theta) \big)}\bigg\|_{L^{\infty}(\mathbb R^N)}Y(t-\theta,\cdot)\ast_x K\big(\varphi(\theta)\big)\,
					\mathrm{d}\theta ~~{\rm for}~ t\in(0,T_1).	    	
				\end{aligned}
			\end{equation}
			From (\ref{eq 4.1}) we deduce that
			\begin{equation}\label{eq 4.4}
				\begin{aligned}
					C&\leq {\varphi(\theta)} \leq K^{-1}\Big((\delta +1)\big\|Z(\theta,\cdot)\ast_x K(u_1) \big\|_{L^{\infty}(\mathbb R^N)}\Big)\\			
					&\leq K^{-1}\Big(M_2(\delta +1)\big[(1\ast l)(\theta) \big]^{-\frac{N}{2c}}\big\|K(u_1) \big\|_{L_{ul}^c(\mathbb R^N)} \Big):=\phi(\theta), \quad 0<\theta<t<T_1.
				\end{aligned}
			\end{equation}
			Due to $f(u)/K(u)$ is nondecreasing for sufficiently large $u$, we obtain			
			\begin{align*}
				\bigg\|\frac{f\big(\varphi(\theta) \big)}{K\big(\varphi(\theta) \big)}\bigg\|_{L^{\infty}(\mathbb R^N)}
				\leq \frac{f\big(\phi(\theta) \big)}{K\big(\phi(\theta) \big)}, \quad 0<\theta<t<T_1.		
			\end{align*}
			This along with (\ref{eq 4.3}), Young's inequality, Lemma \ref{lemma 3.4} and Remark \ref{remark 2.2} gives
			\begin{equation}\label{eq 4.5}
				\begin{aligned}
					&\quad \int_{0}^{t} Y(t-\theta,\cdot)\ast_x f\big(\varphi(\theta)\big)\,
					\mathrm{d}\theta\\
					&\leq \int_{0}^{t} \frac{f\big(\phi(\theta) \big)}{K\big(\phi(\theta) \big)} (\delta +1) Y(t-\theta,\cdot)\ast_x Z(\theta,\cdot) \ast_x K(u_1)\, \mathrm{d}\theta\\
					&\leq \int_{0}^{t} \frac{f\big(\phi(\theta) \big)}{K\big(\phi(\theta) \big)} (\delta +1) \left\|Y(t-\theta,\cdot)\ast_x Z(\theta,\cdot) \ast_x K(u_1)\right\|_{L^{\infty}(\mathbb R^N)} \, \mathrm{d}\theta \\
					&\leq \int_{0}^{t} \frac{f\big(\phi(\theta) \big)}{K\big(\phi(\theta) \big)} (\delta +1) l(t-\theta) \big\|Z(\theta,\cdot)\ast_x K(u_1) \big\|_{L^{\infty}(\mathbb R^N)}\, \mathrm{d}\theta\\
					&\leq \int_{0}^{t} \frac{f\big(\phi(\theta) \big)}{K\big(\phi(\theta) \big)} l(t-\theta) M_2(\delta +1) \big[(1\ast l)(\theta) \big]^{-\frac{N}{2c}} \big\|K(u_1) \big\|_{L_{ul}^c(\mathbb R^N)}\, \mathrm{d}\theta\\
					&\leq \int_{0}^{t} l(t-\theta) f\big( \phi(\theta)\big)\, \mathrm{d}\theta.
				\end{aligned}
			\end{equation}
			Take $\kappa>0$ with $0<{N(\kappa+1)}/{2\rho}<1$. The condition $\lim\limits_{u \to 0} f'(u)G(u)=1$ implies that $f(u)\big[G(u)\big]^{\kappa+1}$ is decreasing for $u\geq C$. Moreover, observe that $C\leq \varphi(t) \leq \phi(t) \leq \phi(\theta)$ for the decreasing property of $\phi$, then we get
			\begin{align*}
				&\quad \int_{0}^{t} l(t-\theta)f\big(\phi(\theta)\big)\, \mathrm{d}\theta\\
				&= \int_{0}^{t} l(t-\theta)f\big(\phi(\theta)\big) \Big[ G\big(\phi(\theta)\big)\Big]^{\kappa+1} \Big[ G\big(\phi(\theta)\big)\Big]^{-\kappa-1}\, \mathrm{d}\theta\\
				&\leq f\big(\varphi(t) \big) G\big(\varphi(t) \big) \big[G(C) \big]^{\kappa} \int_{0}^{t} l(t-\theta) \Big[ G\big(\phi(\theta)\big)\Big]^{-\kappa-1}\,\mathrm{d}\theta.
			\end{align*}
			It follows from Lemma \ref{lemma 3.4} that
			\begin{equation*}
				\begin{aligned}
					&\quad \int_{0}^{t} l(t-\theta) \Big[ G\big(\phi(\theta)\big)\Big]^{-\kappa-1}\, \mathrm{d}\theta\\
					&\leq \int_{0}^{t} l(t-\theta) \Big(M_2 \big[(1\ast l)(\theta) \big]^{-\frac{N}{2c}} \big\|K(u_1) \big\|_{L_{ul}^c(\mathbb R^N)} \Big)^{\frac{c(\kappa+1)}{\rho}}\, \mathrm{d}\theta\\
					&=\Big(M_2 \big\|K(u_1) \big\|_{L_{ul}^c(\mathbb R^N)} \Big)^{\frac{c(\kappa+1)}{\rho}} \int_{0}^{t}l(t-\theta)\big[(1\ast l)(\theta) \big]^{-\frac{N(\kappa+1)}{2\rho}}\, \mathrm{d}\theta\\
					&\lesssim \big[(1\ast l)(t) \big]^{1-\frac{N(\kappa+1)}{2\rho}}\to 0 ~~{\rm as}~ t \to 0.	
				\end{aligned}
			\end{equation*}
			Here, we employ the fact that the nonincreasing and nonnegative property of $l$ implies
			\begin{align}\label{eq 4.6}
				\begin{aligned}
					&\quad \int_{0}^{t}l(t-\theta)\big[(1\ast l)(\theta) \big]^{-\frac{N(\kappa+1)}{2\rho}}\, \mathrm{d}\theta\\
					&\leq \int_{0}^{\frac{t}{2}}l(\theta)\big[(1\ast l)(\theta) \big]^{-\frac{N(\kappa+1)}{2\rho}}\, \mathrm{d}\theta+ \big[(1\ast l)(t) \big]^{-\frac{N(\kappa+1)}{2\rho}}\int_{\frac{t}{2}}^{t}l(t-\theta)\, \mathrm{d}\theta\\
					&\lesssim \big[(1\ast l)(t) \big]^{1-\frac{N(\kappa+1)}{2\rho}}. 
				\end{aligned}
			\end{align}
			Therefore, there exists a sufficiently small $T\in (0,T_1)$ such that
			\begin{align}\label{eq 4.7}
				\varphi(t)-Z(t,\cdot)\ast_x u_0\geq \int_{0}^{t} Y(\theta,\cdot)\ast_x f\big(\varphi(\theta) \big)\, \mathrm{d}\theta, \quad t\in (0,T).
			\end{align}	
			In view of Lemma \ref{lemma 3.1}, there exists $u$ such that $u(t)=\mathcal A\big[u(t) \big]$ with $0\leq u(t) \leq \varphi(t)$ for $t\in (0,T)$ and a.e. $x \in \mathbb R^N$. It follows from $\varphi(t) \in L_{loc}^{\infty}\big((0,T), L^{\infty}(\mathbb R^N) \big)$ that
			$$
			u(t)\in L_{ul}^{1}(\mathbb R^N)~{\rm and}~ f\big(u(t)\big) \in L_{ul}^{1}(\mathbb R^N), \quad t\in (0,T).
			$$
			Moreover, consider $u=K^{-1}(v)$ with $v\geq K(C)$. Since $f(u)\big[G(u) \big]^{\kappa+1}$ is decreasing, then
			$$
			f\big(K^{-1}(v) \big)\lesssim v^{\frac{c(\kappa+1)}{\rho}}, \quad v\geq K(C),
			$$
			which implies
			\begin{align*}
				f\big(u(\theta)\big) \leq f(\varphi(\theta)) \lesssim \big( Z(\theta,\cdot)\ast_x K(u_1)\big)^{\frac{c(\kappa+1)}{\rho}},\quad 0<\theta<t<T.
			\end{align*}
			This along with Lemma \ref{lemma 3.4}, \eqref{eq 4.6}, and $k \ast l \equiv 1$ yields that for $t \in (0,T)$, it holds
			\begin{align*}
				\int_{0}^{t} \Big\|f\big(u(\theta)\big) \Big\|_{L_{ul}^1(\mathbb R^N)}\, \mathrm{d}\theta
				&\lesssim \int_{0}^{t} \Big\|(Z(\theta,\cdot)\ast_x K(u_1))^{\frac{c(\kappa+1)}{\rho}} \Big\|_{L_{ul}^{1}(\mathbb R^N)}\, \mathrm{d}\theta\\
				&\lesssim \int_{0}^{t} \Big\|(Z(\theta,\cdot)\ast_x K(u_1))^{\frac{c(\kappa+1)}{\rho}} \Big\|^{\frac{c(\kappa+1)}{\rho}}_{L_{ul}^{\frac{c(\kappa+1)}{\rho}}(\mathbb R^N)}\, \mathrm{d}\theta\\
				&\lesssim \int_{0}^{t} \Big\|(Z(\theta,\cdot)\ast_x K(u_1))^{\frac{c(\kappa+1)}{\rho}} \Big\|^{\frac{c(\kappa+1)}{\rho}}_{L^{\infty}(\mathbb R^N)}\, \mathrm{d}\theta\\
				&\lesssim \int_{0}^{t} \big[(1\ast l)(\theta) \big]^{-\frac{N(\kappa+1)}{2\rho}}\, \mathrm{d}\theta\\
				&\lesssim (k\ast l \ast((1\ast l)^{-\frac{N(\kappa+1)}{2\rho}} ))(t)\\
				&\lesssim (1\ast k)(T)\big[(1\ast l)(T) \big]^{1-\frac{N(\kappa+1)}{2\rho}} < \infty.
			\end{align*}
			
			Finally, we obtain that
			\begin{align*}
				\big\|u(t,\cdot)- Z(t,\cdot)\ast_x u_0\big\|_{L^{\infty}(\mathbb R^N)}
				&\leq \Big\|\int_{0}^{t} Y(t-\theta,\cdot)\ast_x f\big(\varphi(\theta)\big)\, \mathrm{d}\theta \Big\|_{L^{\infty}(\mathbb R^N)}\\
				&\leq \int_{0}^{t} l(t-\theta) f\big( \phi(\theta)\big)\, \mathrm{d}\theta\\
				&\leq f(C) \big[G(C) \big]^{\kappa+1} \int_{0}^{t} l(t-\theta) \Big[ G\big(\phi(\theta)\big)\Big]^{-\kappa-1}\,\mathrm{d}\theta\\
				&\lesssim \big[(1\ast l)(t) \big]^{1-\frac{N(\kappa+1)}{2\rho}}\to 0 ~~{\rm as}~ t \to 0.
			\end{align*}
			The proof is completed.	
		\end{proof}	
	\end{theorem}
	For the critical case $\rho = \frac{N}{2}$, we have the following result.
	\begin{theorem}\label{the 4.2}
		Let $N\geq 1$, $\rho = \frac{N}{2}$, $(k,l)\in (\mathcal {PC})$ and $u_0 \geq 0$. Suppose that $f$ satisfies $(A_f)$. If the following (i), (ii) and (iii) hold:
		\begin{enumerate}
			\item[\rm (i)]
			${G(u_0)}^{-\rho}\in \mathcal L_{ul}^1(\mathbb R^N)$.
			\item[\rm (ii)]
			$f'(u)G(u)\leq 1$ for sufficiently large $u>0$.
			\item[\rm (iii)]
			There exists a constant $T_2>0$ such that
			\begin{align}\label{eq 4.8}
				\sup_{t\in(0, T_2)}\int_{0}^{t}l(t-\theta)\big[(1\ast l)(\theta) \big]^{-1}\, \mathrm{d}\theta <\infty.
			\end{align}
		\end{enumerate}
		Then \eqref{1.2} admits a nonnegative mild solution on the interval $[0, T)$ for some $T > 0$.
	\end{theorem}
	\begin{proof}
		We continue to use the definitions of $u_1$, $K$ and $\varphi$ in the proof of Theorem \ref{the 4.1}. By \cite[Lemma 2.5]{M. Suzuki 19} we get $K(u_1)\in \mathcal L_{ul}^c(\mathbb R^N)$. It follows from Lemma \ref{lemma 3.5} that there is a nonnegative and increasing function $C^{\ast}(t)$ for $t \in (0,1)$ with $\lim\limits_{t\to 0}C^{\ast}(t)=0$, such that
		\begin{align*}
			\big\|Z(t,\cdot)\ast_x K(u_1) \big\|_{L^{\infty}(\mathbb R^N)}
			\leq C^{\ast}(t)\big[(1\ast l)(t) \big]^{-\frac{N}{2c}} ~~{\rm for}~ t \in (0,1).
		\end{align*}
		{	Let $t$  be sufficiently small, then we obtain
			\begin{align*}
				C
				&\leq \varphi(\theta) \leq K^{-1}\Big((\delta +1)C^{\ast}(\theta)\big [(1\ast l)(\theta) \big]^{-\frac{N}{2c}} \Big):=\phi(\theta) ~~{\rm for}~ \theta\in (0,t).
		\end{align*}}
		Similar to  (\ref{eq 4.5}), it holds that
		\begin{equation*}
			\begin{aligned}
				&\quad \int_{0}^{t} Y(t-\theta,\cdot)\ast_x f\big(\varphi(\theta)\big)\,
				\mathrm{d}\theta\\
				&\leq \int_{0}^{t} \frac{f\big(\phi(\theta) \big)}{K\big(\phi(\theta) \big)} l(t-\theta) (\delta +1)C^{\ast}(\theta) \big[(1\ast l)(\theta) \big]^{-\frac{N}{2c}}\, \mathrm{d}\theta\\
				&\leq \int_{0}^{t} l(t-\theta) f\big( \phi(\theta)\big)\, \mathrm{d}\theta.	
			\end{aligned}
		\end{equation*}
		Since $f'(u)G(u)\leq 1$ for large $u>0$, then $f(u)G(u)$ is nonincreasing on $[C, \infty)$. The monotonicity of $C^{\ast}(t)$ , $\lim\limits_{t\to 0}C^{\ast}(t)=0$ and (\ref{eq 4.8}) give
		\begin{align*}
			&\quad \int_{0}^{t} Y(t-\theta,\cdot)\ast_x f\big(\varphi(\theta)\big)\,
			\mathrm{d}\theta\\
			&\leq f\big(\varphi (t)\big)G\big(\varphi (t)\big) \int_{0}^{t} l(t-\theta) \Big[ G\big(\phi(\theta)\big)\Big]^{-1}\, \mathrm{d}\theta\\
			&\lesssim f\big(\varphi (t)\big)G\big(\varphi (t)\big)C^{\ast}(t)\int_{0}^{t} l(t-\theta) \big[ (1*l)(\theta)\big]^{-1}\, \mathrm{d}\theta\\
			&\lesssim C^{\ast}(t)\to 0 ~~{\rm as}~t \to 0,		
		\end{align*}
		which indicates (\ref{eq 4.7}) holds for
		sufficiently small $T>0$. Similarly, there exists a function $u$ such that $u(t)=\mathcal A\big[u(t) \big]$ for $t \in (0,T)$ and a.e. $x \in \mathbb R^N$, $u(t)\in L_{ul}^{1}(\mathbb R^N)$ and $f\big(u(t)\big) \in L_{ul}^{1}(\mathbb R^N)$ for $t \in (0,T)$.	
		
		Moreover, consider $u=K^{-1}(v)$ with $v\geq K(C)$. Since $f(u)G(u)$ is decreasing, then the following inequality holds:
		$$
		f\big(K^{-1}(v) \big)\lesssim v^{\frac{c}{\rho}}, \quad v\geq K(C),
		$$
		which implies
		\begin{align*}
			f\big(u(s)\big) \leq f(\varphi(s)) \lesssim \big( Z(t,\cdot)\ast_x K(u_1)\big)^{\frac{c}{\rho}},\quad s\in(0,t).
		\end{align*}
		Similar to the proof of Theorem \ref{the 4.1}, we have
		\begin{align*}
			\int_{0}^{t} \Big\|f\big(u(\theta)\big) \Big\|_{L_{ul}^1(\mathbb R^N)}\, \mathrm{d}\theta
			&\lesssim \int_{0}^{t} \Big\|\big(Z(\theta,\cdot)\ast_x K(u_1)\big)^{\frac{c}{\rho}} \Big\|_{L_{ul}^{1}(\mathbb R^N)}\, \mathrm{d}\theta\\
			&\lesssim \int_{0}^{t} \big\|Z(\theta,\cdot)\ast_x K(u_1) \big\|^{\frac{c}{\rho}}_{L_{ul}^{\frac{c}{\rho}}(\mathbb R^N)}\, \mathrm{d}\theta\\
			&\lesssim \int_{0}^{t} \big\|Z(\theta,\cdot)\ast_x K(u_1) \big\|^{\frac{c}{\rho}}_{L^{\infty}(\mathbb R^N)}\, \mathrm{d}\theta\\
			&\lesssim \big[C^{\ast}(t)\big]^{\frac{2c}{N}}\int_{0}^{t} \big[(1\ast l)(\theta) \big]^{-1}\, \mathrm{d}\theta \\
			&\lesssim \big[C^{\ast}(t)\big]^{\frac{2c}{N}}(k \ast l \ast (1\ast l)^{-1})(t)\, \\
			&\lesssim \big[C^{\ast}(T)\big]^{\frac{2c}{N}}(1\ast k)(T)< \infty ~~{\rm for}~ t \in (0,T),
		\end{align*}
		and $\big\|u(t,\cdot)- Z(t,\cdot)\ast_x u_0\big\|_{L^{\infty}(\mathbb R^N)}\lesssim C^{\ast}(t)\to 0 ~~{\rm as}~t \to 0$.
	\end{proof}
	
	\begin{remark}\label{rem 4.1}
		Condition (iii) of Theorem \ref{the 4.2} holds for many, but not all, pairs $(k,l) \in (\mathcal{PC})$. To demonstrate this, Section 5 provides two examples of pairs $(k,l) \in (\mathcal{PC})$ that violate this condition.
	\end{remark}
	
	In addition, we point out that the following estimate
	\begin{align}\label{eq 4.9}
		[(1\ast l)(t)]^{-1}\lesssim k(t) ~~{\rm as}~ t \to 0
	\end{align}
	holds for {some pairs} $(k,l)\in (\mathcal {PC})$. {Let $0\leq \nu<1$, it follows
		\begin{align*}
			\int_{0}^{t}l(t-\theta)\big[(1\ast l)(\theta) \big]^{\nu-1}\, \mathrm{d}\theta \lesssim \left[(1\ast l)(t)\right]^\nu \int_{0}^{t} l(t-\theta)k(\theta)\, \mathrm{d}\theta \leq [(1\ast l)(t)]^{\nu}
		\end{align*}
		and
		\begin{align*}
			\int_{0}^{t}\big[(1\ast l)(\theta) \big]^{\nu-1}\, \mathrm{d}\theta \lesssim \big[(1\ast l)(t)\big]^{\nu}(1 \ast k)(t)
		\end{align*}
		for sufficiently small $t>0$.} Likewise, we have the following corollary.
	\begin{corollary}\label{cor 4.2}
		Let $N\geq 1$, $(k,l)\in (\mathcal {PC})$ and $u_0 \geq 0$. Suppose that (\ref{eq 4.8}) holds, and $f$ satisfies $(A_f)$. 
		If any of the following conditions is satisfied:
		\begin{enumerate}
			\item[\rm (i)]
			${G(u_0)}^{-\rho}\in L_{ul}^1(\mathbb R^N)$ for $\rho > \frac{N}{2}$.
			\item[\rm (ii)]
			${G(u_0)}^{-\rho}\in \mathcal L_{ul}^1(\mathbb R^N)$ for $\rho = \frac{N}{2}$ and
			$f'(u)G(u)\leq 1$ for $u>0$.
		\end{enumerate}
		Then \eqref{1.2} admits a nonnegative mild solution on the interval $[0, T)$ for some $T > 0$.
	\end{corollary}
	
	\subsection{Nonexistence}
	
	In this subsection, we give a sufficient condition for the equation
	(\ref{1.2}) to be unsolvable  in the supercritical case $\rho \in (0, \frac{N}{2})$.
	\begin{theorem}\label{the 4.3}
		Let $N\geq 1$ and $(k,l)\in(\mathcal {PC})$. Suppose that $f$ satisfies $A_f$. Then there is an initial function $u_0\geq 0$ satisfying $G(u_0)^{-\rho}\in L_{ul}^{1}(\mathbb R^N)$ for any $\rho \in (0,\frac{N}{2})$, such that (\ref{1.2}) does not have a nonnegative mild solution $u$ on the interval $[0, T)$ for any $T>0$.
	\end{theorem}
	\begin{proof}
		Define
		\begin{align*}
			K_1(u):=
			\left\{\begin{aligned}
				&G^{-1}(u), \quad 0<u<\lim\limits_{u\to 0}G(u)=G(0),\\
				&0,\quad u\geq G(0), \end{aligned}\right.
		\end{align*}
		then $K_1(u)$ is a nonincreasing and  convex function on $(0,\infty)$. Take $\sigma$ such that $2<\sigma<\frac{N}{\rho}$. Define $u_0(x):=K_1(|x|^\sigma)$, then $G(u_0)^{-\rho}\in L_{ul}^1(\mathbb R^N)$.
		
		Suppose that (\ref{1.2}) has a local mild solution in the sense of Definition \ref{def 2.1} on the interval $[0,T)$ for a certain constant $T>0$. Using (iii) of Definition \ref{def 2.1} we obtain $Z(t,\cdot)\ast_x u_0<u(t)$ in  $(0,T)\times \mathbb R^N$. This along with (v) of Definition \ref{def 2.1} shows that
		\begin{align}\label{eq 4.10}
			\int_{0}^{t} Y(t-\theta,\cdot)\ast_x f\big(Z(\theta,\cdot)\ast_x u_0 \big)\, \mathrm{d}\theta \leq \int_{0}^{t} Y(t-\theta,\cdot)\ast_x f\big( u(\theta)\big)\, \mathrm{d}\theta \to 0
		\end{align}
		in $L^ {\infty}(\mathbb R^N)$ as $t \to 0$.
		
		Let $u_1(x) = |x|^\sigma$. We now estimate the upper bound of $Z(\theta,\cdot) \ast_x u_1$ for $x \in [(1 \ast l)(t)]^{1/2}$ with $\theta \in (0,t)$. Note that
		\begin{align*}
			&\quad \int_{\mathbb R^N} H(\tau,x-y) u_1(y)\, \mathrm{d}y\\
			&= \int_{|y| \leq 2[(1 \ast l)(t)]^{\frac{1}{2}  }} H(\tau,x-y)  u_1(y)\, \mathrm{d}y + \int_{|y| > 2[(1 \ast l)(t)]^{\frac{1}{2}  }} H(\tau,x-y) u_1(y)\, \mathrm{d}y\\
			&=J_1(\tau,x) + J_2(\tau,x). 
		\end{align*}
		It is clear that
		\begin{align}\label{eq 4.11}
			\begin{aligned}
				J_1(\tau,\cdot) 
				\leq 2^\sigma[(1\ast l)(t)]^{\frac{\sigma}{2} } \sup_{|x|\leq [(1\ast l)(t)]^{\frac{1}{2} }}\int_{\mathbb R^N} H(\tau,x-y)\, \mathrm{d}y
				\leq 2^\sigma[(1\ast l)(t)]^{\frac{\sigma}{2}}
			\end{aligned}
		\end{align}
		for $|x|\leq [(1\ast l)(t)]^{\frac{1}{2} } $. We next estimate $J_2(t,x)$. Observe that the conditions $|x| \leq [(1 \ast l)(t)]^{1/2}$ and $|y| > 2[(1 \ast l)(t)]^{1/2}$ imply $|x-y| \geq \big||y| - |x|\big| \geq \frac{|y|}{2}$. Consequently,
		\begin{align}\label{eq 4.12}
			\begin{aligned}
				J_2(\tau,x) 
				&=(4\pi \tau)^{- \frac{N}{2}} \int_{ |y| > 2[(1 \ast l)(t)]^{\frac{1}{2}  } } e^{ \frac{-|x-y|^2}{4\tau} }|y|^\sigma \, \mathrm{d}y\\
				&\lesssim \tau^{-\frac{N}{2} } \int_{ |y| > 2[(1 \ast l)(t)]^{\frac{1}{2}  } } e^{ \frac{-|y|^2}{8\tau} }|y|^\sigma \, \mathrm{d}y\\
				&\lesssim \tau^{-\frac{N}{2} } \int_0^\infty e^{ \frac{-r^2}{8\tau} }r^{N-1+\sigma} \, \mathrm{d}r \lesssim \tau^{ \frac{\sigma}{2} }.  
			\end{aligned}  
		\end{align}
		Combining \eqref{eq 2.3}, \eqref{eq 4.11}, and \eqref{eq 4.12} yields
		\begin{align*}
			Z(\theta,\cdot)\ast_x u_1 &=- \int_0^\infty \int_{\mathbb R^N} H(\tau,x-y) u_1(x)\,\mathrm{d}y \mathrm{d}\tau =
			-\int_0^\infty [J_1(\tau,x) +J_2(\tau,x) ] w(t,\mathrm{d}\tau) \, \mathrm{d}\tau\\
			&\lesssim [(1\ast l)(t)]^{\frac{\sigma}{2} } 
			+c_{\frac{\sigma}{2} }(t),
		\end{align*}
		where $c_{\frac{\sigma}{2}}(t):= -\int_0^\infty \tau^{ \frac{\sigma}{2}} w(t,\mathrm{d}\tau)$. Clearly, there exists a unique positive integer $n_1$ with $\frac{N}{2\rho} \in (n_1, n_1+1]$. Since $2<\sigma< \frac{N}{\rho}$ indicates $\frac{\sigma}{2(n_1+1)}\in (0,1)$.   
		An application of Lemma \ref{lemma 2.1} and Lemma \ref{lemma 3.3} gives
		\begin{align*}
			\widehat{c}_{\frac{\sigma}{2} }(\lambda) &= \int_0^\infty \tau^{ \frac{\sigma}{2} } (\lambda \widehat{l}(\lambda))^{-1} e^{ -\tau\widehat{l}(\lambda)^{ -1} } \, \mathrm{d}\tau =  \int_0^\infty \tau^{ \frac{\sigma}{2} } (\lambda \widehat{l}(\lambda))^{-1} e^{ -\tau\widehat{l}(\lambda)^{ -1} } \, \mathrm{d}\tau\\
			&= \Gamma(1+\frac{\sigma}{2}) \lambda^{-1} \widehat{l}(\lambda)^{ \frac{\sigma}{2} } = \Gamma(1+\frac{\sigma}{2}) \lambda^{-1} \left[\widehat{l}(\lambda)^{ \frac{\sigma}{2(n_1+1)}  } \right]^{n_1+1}\\
			&=\Gamma(1+\frac{\sigma}{2}) \lambda^{-1} \widehat{b}_{1-  \frac{\sigma}{2(n_1+1)}  }(\lambda)^{n_1+1} 
		\end{align*}
		By the uniqueness of the Laplace transform, we have
		\begin{align*}
			c_{ \frac{\sigma}{2} }(t) = \Gamma(1+\frac{\sigma}{2}) (1 \ast \underbrace{ 	b_{1-  \frac{\sigma}{2(n_1+1)}} \ast  \cdots \ast b_{1-  \frac{\sigma}{2(n_1+1)}}  }_{ n_1+ 1 \text{ terms }} )(t),  
		\end{align*}
		which implies $c_{\frac{\sigma}{2}}(t) \lesssim [(1\ast b_{1-\frac{\sigma}{2(n_1+1)} })(t)]^{ n_1+1 }$. Applying Lemma \ref{lemma 3.3} again yields
		\begin{align*}
			c_{\frac{\sigma}{2}}(t) \lesssim \left[\int_0^t l(\theta)(1\ast l)(\theta)^{ \frac{\sigma}{2(n_1+1)} -1} \, \mathrm{d}\theta \right]^{ n_1+1 } \lesssim [(1 \ast l)(t)]^{\frac{\sigma}{2} }. 
		\end{align*}
		Therefore, we have
		\begin{align}\label{eq 4.13}
		    Z(\theta,\cdot)\ast_x u_1 \leq C[(1 \ast l)(t)]^{\frac{\sigma}{2} } \text{ ~for~ } 0<\theta<t, \ x \in [(1\ast l)(t)]^{ \frac{1}{2} },  
		\end{align}
		where $C$ is sufficiently large. 
		
		Since $K_1$ is nonincreasing and convex, it follows from Lemma \ref{lemma 3.6} that for $\theta \in (0,t)$ and $|x|<\big[(1\ast l)(t) \big]^{\frac{1}{2}}$,
		\begin{align}\label{eq 4.14}
				\begin{aligned}
					Z(\theta,\cdot)\ast_x u_0
					&= Z(\theta,\cdot)\ast_x K_1 ( u_1) \geq K_1\left(Z(\theta,\cdot)\ast_x u_1 \right)\\
					&\geq K_1\left(\left\|Z(\theta,\cdot)\ast_x (u_1) \right\|_{L^{\infty}\Big(|x|<\big[(1\ast l)(t) \big]^{\frac{1}{2}}\Big)} \right)\\
					&\geq K_1\Big(C\big[(1\ast l)(t) \big]^{\frac{\sigma}{2}} \Big).
				\end{aligned}
		\end{align}
		Take $\kappa>0$ such that $\kappa+2<\sigma$. For sufficiently large $u>0$, $f(u)\big[G(u) \big]^{\frac{\kappa+2}{\sigma}}$ is increasing and $f(u)\big[G(u) \big]^{\frac{\kappa+2}{\sigma}}\gtrsim 1$. It is enough to prove that  $v^{-\frac{\kappa+2}{\sigma}} \lesssim f\circ G^{-1}(v)$ for sufficiently small $v$. {We choose $t>0$ small enough to obtain}
		\begin{align*}
			f\circ K_1\Big(C\big[(1\ast l)(t) \big]^{\frac{\sigma}{2}} \Big)
			=f\circ G^{-1}\Big(C\big[(1\ast l)(t) \big]^{\frac{\sigma}{2}} \Big)
			\gtrsim \big[(1\ast l)(t) \big]^{-\frac{\kappa+2}{2}}.
		\end{align*}
		This together with  (\ref{eq 4.14}), Remark \ref{remark 2.2} shows that
		\begin{align*}
			Y(t-\theta,\cdot)\ast_x f\big(Z(\theta,\cdot)\ast_x u_0 \big)
			&\geq Y(t-\theta,\cdot)\ast_x \left(f\circ K_1
			\Big(C\big[(1\ast l)(t) \big]^{\frac{\sigma}{2}} \Big)\right)
			\\
			&=l(t-\theta) f\circ
			K_1\Big(C \big[(1\ast l)(t) \big]^{\frac{\sigma}{2}} \Big)\\
			&\gtrsim l(t-\theta)\big[(1\ast l)(t) \big]^{-\frac{\kappa+2}{2}}.	
		\end{align*}
		Obviously $\lim\limits_{t\to 0}(1\ast l)(t)=0$, then we have
		\begin{align*}
			\int_{0}^{t} Y(t-\theta,\cdot)\ast_x f\big(Z(\theta,\cdot)\ast_x u_0 \big)\, \mathrm{d}\theta
			&\gtrsim \int_{0}^{t} l(t-\theta)\big[(1\ast l)(t) \big]^{-\frac{\kappa+2}{2}}\, \mathrm{d}\theta\\
			&=\big[(1\ast l)(t) \big]^{-\frac{\kappa}{2}}\to \infty ~~{\rm as}~ t \to 0.
		\end{align*}
		This is in contradiction with (\ref{eq 4.10}). Therefore, for any $T>0$, (\ref{1.2}) does not exist a nonnegative mild solution.
	\end{proof}
	
	\section{Applications}
	In this section, we give some examples to illustrate the wide applicability of the obtained results. For a general pair $(k,l) \in (\mathcal {PC})$, it is difficult to verify the condition {\rm(iii)} in Theorem \ref{the 4.2}. To do this, we only need the short-time behavior of $k$, $l$ and $(1\ast l)(t)$, which can be established by the Karamata-Feller Tauberian theory,
	it claims that the short-time (long-time) behavior of a function $g(t)$ is determined by the behavior of its Laplace transform $\widehat {g}(\lambda)$ as $\lambda \to \infty$ ($\lambda \to 0$). Thus we introduce a kind of Karamata-Feller Tauberian theory, see \cite[Chapter XIII]{W. Feller}.
	
	\begin{lemma}\label{lemma 5.1}
		Let $L:(0, \infty) \to (0, \infty)$ denote a function that is slowly varying at $\infty$, that is, for every fixed $\varsigma>0$ we have $L(t\varsigma)/L(t) \to 1$ as $t \to \infty$. Let $\beta>0$ and $g:(0, \infty) \to R$ is denoted as a monotone function whose Laplace transform $\widehat{g}(\lambda)$ exists for all $\lambda \in C_{+}:=\{\lambda \in C:Re \ \lambda >0\}$. Then
		$$
		\widehat{g}(\lambda) \sim \frac{1}{\lambda^{\beta}}L(\lambda) ~~{as}~ \lambda \to \infty
		$$
		if and only if
		$$
		g(t) \sim \frac{t^{\beta - 1}}{\varGamma(\beta)} L\left(\frac{1}{t}\right) ~~{as}~ t \to 0.
		$$
	\end{lemma}
	
	Next, we give the several examples of type $(\mathcal {PC})$. Although the asymptotic behavior of these examples may already be available in the literature, we provide detailed calculations for the sake of completeness and self-containedness of this paper.
	
	{\bf Example 5.1}	
	Let $0<\alpha<\beta<1$. Consider the kernel
	$$
	k(t)=g_{1-\alpha}(t) + g_{1-\beta}, \ t>0.
	$$
	There exists a function $l \in L_{loc}^{1}(\mathbb R_{+})$ such that $(k,l) \in (\mathcal {PC})$, see \cite[Example 2]{J.C. Pozo}. Moreover, we can see that
	$$
	\widehat l(\lambda) =  \frac{1}{\lambda^{\alpha} + \lambda^{\beta}}, \quad \lambda>0.
	$$
	Since $0<\alpha<\beta<1$, we obtain
	$$
	\widehat{(1\ast l)}(\lambda) \sim \frac{1}{\lambda^{1+\beta}} ~~{\rm as}~ \lambda \to \infty.
	$$
	By Lemma \ref{lemma 5.1}, we get
	$$
	(1\ast l)(t) \sim t^{\beta} ~~{\rm as}~ t \to 0.
	$$
	Note that $k(t)\sim t^{-\beta}$ as $t \to 0$.
	The above conclusion can also be extended to kernels that have the form $k(t)=\sum_{k=1}^{n} c_j g_{1-\alpha_j}(t)$, where $c_j>0$ and  $0<\alpha_1<\alpha_2<\cdots<\alpha_n<1$.
	
	{\bf Example 5.2}
	The time fractional case with weight. Let $0<\alpha<\beta<1$. Consider the kernel
	$$
	k(t)=g_\beta(t)E_{\alpha,\beta}(-\eta t^\alpha), \quad t>0.
	$$
	There exists a function $l \in L_{loc}^{1}(\mathbb R_{+})$ such that $(k,l) \in (\mathcal {PC})$, see \cite[Example 3]{J.C. Pozo}. Moreover, we have
	$$
	\widehat k(\lambda)=\frac{1}{\lambda^{\beta}} L_1\left(\lambda \right)~~{\rm and}~~ \widehat{(1\ast l)}(\lambda)=\frac{1}{\lambda^{2-\beta}}L_2(\lambda), \quad \lambda>0,
	$$
	where $L_1(\lambda)=\lambda^\alpha/(\lambda^\alpha+\eta)$, $L_2(\lambda)=\left[L_1(\lambda) \right]^{-1}$, for $\lambda>0$. $L_1(t)$ and $L_2(t)$ is slowly varying at $t=\infty$. By $0<\alpha<\beta<1$ and Lemma \ref{lemma 5.1} we obtain
	$$
	k(t)\sim t^{\beta -1} ~~{\rm and}~~ (1\ast l)(t)\sim t^{1-\beta} ~~{\rm as}~~ t\to 0.
	$$
	
	{\bf Example 5.3}
	Let $\gamma>0$ and $\alpha \in (0,1)$. Consider the following kernels
	$$
	k(t)=g_{1-\alpha}(t)e^{-\gamma t}+\gamma \int_{0}^{t} g_{1-\alpha}(\theta)e^{-\gamma \theta}\, \mathrm{d}\theta ~~{\rm and}~~ l(t)=g_\alpha(t)e^{-\gamma t}, \quad t>0.
	$$
	It follows from \cite[Example 3.5]{F. Alegria} that $(k,l) \in (\mathcal {PC})$. Direct calculation yields
	$$
	k(t) \sim g_{1-\alpha}(t) ~~{\rm and}~~ l(t) \sim g_\alpha(t) ~~{\rm as}~ t \to 0.
	$$
	Obviously, the kernel pairs $(k,l)$ introduced above meet the requirements of all the theorems and corollary in Section 4.
	
	Below we present another interesting example. 
	
	{\bf Example 5.4}
	Let $n\in \mathbb N$. Consider the kernel
	$$
	k(t)=\int_{0}^{t}g_\alpha(t)\alpha^n\, \mathrm{d}\alpha, \quad t>0.
	$$
	It follows from \cite[Example 4]{J.C. Pozo} that there exists a monotone decreasing function $l \in L_{loc}^{1}(\mathbb R_{+})$ such that $(k,l) \in (\mathcal {PC})$, and
	$$
	\widehat l(\lambda) =\frac{1}{\lambda} L_1(\lambda) ~~{\rm and}~~ \widehat{(1\ast l)}(\lambda) =\frac{1}{\lambda^2} L_1\left(\lambda \right),\quad \lambda>0,
	$$
	where
	\begin{align*}
		L_1(\lambda)=\frac{\log^{n+1}(\lambda)}{n!\left(1 - \sum_{m=0}^n \frac{\log^m(\lambda)}{m!\lambda} \right)}, \quad \lambda>0.
	\end{align*}
	We can see that $L_1(t)$ is slowly varying at $t=\infty$. By Lemma \ref{lemma 5.1}, we obtain
	$$
	l(t)\sim \left|\log^{n+1}(t)\right| ~{\rm and}~ (1\ast l)(t)\sim t \left|\log^{n+1}(t)\right| ~~{\rm as}~ t\to 0.
	$$
	By direct calculation, we obtain
	\begin{align*}
		\int_{0}^{t} l(t-\theta)\left[(1 \ast l)(\theta)\right]^{-1}\, \mathrm{d}\theta
		\gtrsim \int_{0}^{\frac{t}{2}} \frac{\left|\log^{n+1}(t-\theta)\right|}{\theta \left|\log^{n+1}(\theta)\right|}\, \mathrm{d}\theta \gtrsim \left|\log(t)\right| ~~{\rm as}~ t \to 0
	\end{align*}	
	for $n \geq 1$, and $\left(l \ast(1 \ast l)^{-1}\right)$ is not well-defined
	on $[0,T)$ with any $T \in (0, \infty)$ for $n=0$. Note that
	\begin{align*}
		k(t)=\int_{0}^{1} \frac{t^{\alpha-1}\alpha^{n+1}}{\varGamma(1+\alpha)}\, \mathrm{d}\alpha \asymp \int_{0}^{1} t^{\alpha-1} \alpha^{n+1}\, \mathrm{d}\alpha ~~{\rm as}~ t \to 0.
	\end{align*}
	Using the integral-by-parts formula, we obtain
	\begin{align*}
		k(t)\asymp \frac{(-1)^{n+2}(n+1)!}{t \log^{n+2}(t)}+ \frac{(n+1)!}{\log^{n+2}(t)}\sum_{m=0}^{n+1} \frac{(-1)^{n+1-m} \log^m(t)}{m!},					
	\end{align*}
	which yields $ k(t) \asymp t^{-1} \left| \log^{n+2}(t) \right|^{-1} ~~{\rm as}~ t \to 0$. Obviously, such kernel pairs do not satisfy \eqref{eq 4.9}.
	
	The kernel $k$ in Example 5.4 is a distributed-order kernel, which characterizes the memory effects in ultraslow diffusion processes. In this case, we can only apply Theorem \ref{the 4.1} to obtain the local existence in the subcritical regime ($\rho>\frac{N}{2}$), and use Theorem \ref{the 4.3} to derive the nonexistence result in the supercritical regime ($\rho\in(0, \frac{N}{2})$). The solvability of problem \eqref{1.2} in the critical case $\rho = \frac{N}{2}$ remains open and requires further investigation.
	
	\section*{Statements and Declarations }
	\noindent{\bf Competing interests}\\
	The authors declare that there is no conflict of interests regarding the publication of this paper.\\
	\noindent{\bf Funding}\\
	This work was supported by National Natural Science Foundation of China (12071396) and the Project of Hunan Provincial Education Department of China (23B0125).\\
	\noindent{\bf Availability of data and materials}\\
	Not applicable.

\end{document}